\documentclass[12pt]{article}

\usepackage{amsmath,amsfonts,amssymb,amsthm, mathrsfs}
\usepackage{fullpage}
\usepackage{stmaryrd}
\usepackage{hyperref}
\theoremstyle{definition}
\newtheorem{definition}{Definition}
\newtheorem{remark}[definition]{Remark}

\newtheoremstyle{mytheorem}{0.5cm}{0.2cm}{\slshape}{ }{\bfseries}{.}{ }{}
\theoremstyle{mytheorem}
\newtheorem{theorem}[definition]{Theorem}

\newtheorem{proposition}[definition]{Proposition}

\renewcommand{\u}{\mathsf{u}}
\renewcommand{\v}{\mathsf{v}}

\newcommand{\e}{\mathsf{e}}
\newcommand{\n}{\mu }
\newcommand{\X}{\mathcal{X}}

\newcommand{\E}{\mathbf{E}}
\renewcommand{\P}{\mathbf{P}}

\renewcommand{\phi}{\varphi}
\renewcommand{\kappa}{\varkappa}

\newcommand{\ext}{{\rm ext}}

\newcommand{\cone}{{\rm cone}}

\newcommand{\LM}{\cite{LacMol11}}
\newcommand{\comp}{\cite{Lac}}

\newcommand{\U}{\mathcal{U}}

\newcommand{\sE}{\mathsf{E}}

\renewcommand{\U}{\mathscr{U}}

\newcommand{\V}{\mathscr{V}}

\newcommand{\F}{\mathscr{F}}

\newcommand{\B}{\mathsf{B}}

\newcommand{\uc}{unit covariance}
\newcommand{\iv}{indicator covariogram}

\newlength{\querylen}
\usepackage{fancybox}

\numberwithin{equation}{section}
\numberwithin{definition}{section}
\author{Rapha\"el Lachi\`eze-Rey}

\date{ }

\begin{document}
\bibliographystyle{plain}

\title{The convex class of realisable unit covariances} 

\maketitle

%elucidate?

%combinatorial?

\begin{abstract}

This paper concerns the characterisation of second order marginals for random sets in a discrete setting. Under the instance of unit covariances, this problem possesses a combinatorial symmetry, exploited jointly in the companion paper to give a heuristic procedure to check realisability.  In particular we disprove Matheron's conjecture, and explicit partially the structure of the convex body formed by realisable unit covariances in a finite set.

\end{abstract}


 
\section{Introduction}
 
 Let $\X$ be a set, and call  $\B_{\X}=\{-1,1\}^{\X}$ the class of binary functions on $\X$. Call \emph{unit field}  on $\X$ a random element $X$ in $ \B_{\X}$ The \emph{unit covariance} of $X$ is then defined by 
\begin{align}
\label{eq:def-uc}
\rho^{X}_{x,y}=\E X_xX_y=2\P(X_x=X_y)-1,\quad x,y\in  \X.
\end{align}

 %We prefer this definition, also called \emph{two point covering function} in \cite{Molchanov}, to its centred version $\P(x,y\in X)-\P(x\in X)\P(y\in X)$, as the first order marginal $x\mapsto \P(x\in X)$ is not relevant in this paper. 
The central question here is the inverse \emph{realisability} problem, given a bivariate function $(\rho  _{x,y}; ~x,y\in \X)$ in the class $\F_{\X}$ of symmetric functions on $\X$, to check whether it can be realised by some unit field $X$ (i.e. $\rho =\rho^{X}$). The typical example of a non-realisable, or non-admissible, symmetric function $\rho $ is, with $\X=\{1,2,3\}$, 
\begin{align*}
\rho _{1,2}=1,\quad\rho _{1,3}=1,\quad \rho _{2,3}=-1,\quad \rho _{i,i}=1;i=1,2,3,
\end{align*}
because if $\rho =\rho^{X}$ for some unit field $X$, then $X_1=X_2$ a.s., $X_2=X _{3}$ a.s., but $X_1=-X _{3}$ a.s..

Measurability issues don't matter in this paper, so consider that $\B_{\X}$ is endowed with the discrete topology.
The affine transformation $X\mapsto 2X-1$ transfers unit fields to $\{0,1\}$-valued processes, assimilable to random sets. Such marginal problems arise in many contexts,  the article \cite{FriCha12}, for instance,  calls \emph{contextual} such an admissible function, in the sense that it can be inserted in a real physical context, with applications in the fields of information theory, game theory, quantum mechanics. The authors use an entropic approach to solve some particular related questions.  Our goals turn more towards materials science and geostatistics, where such a characterisation could serve many purposes in modelisation and estimation, see the companion paper \comp~for more details. On a more theoretical level,  characterising the class of second order characteristics would provide insights on a spectral theory for random sets, which apart from Koch et al.   \cite{KocOhsSch03} is currently a gap in the literature.  Matheron identified a combinatorially compact description of the class of realisable unit covariances if $card(\X)\leq  6$, highlighting an elegant combinatorial structure. \\

  Call $\U_{\X}$ the class of  admissible unit covariances on $\X\times \X$. 
It is clear under the form (\ref{eq:def-uc}) that to belong to $\U_\X$ a function $\rho $ has to be semi-definite positive, and lie in the space $\F_{\X}^{1}$ of functions taking the value $1$ on the diagonal ($\rho _{x,x}=1$ for all $x\in \X$), but those conditions are not sufficient.
Let  $ \F_{\X}'$ be the space of symmetric functions with finite support, and define for $\alpha \in \F_{\X}'$  
\begin{align*}
\mathbf{g}_{\alpha }({\u} )=\sum_{x,y\in \X}\alpha _{x,y}{\u} _x{\u} _y; {\u}\in \B_{\X}.
\end{align*} A necessary and sufficient condition for $\rho \in \F_{\X}$ to be a \uc~is that for every  function $\alpha \in \F_{\X}'$, 
\begin{align}
\label{eq:intro-positivity}
\sum_{x,y\in \X}\alpha _{x,y}\rho _{x,y}\geq \kappa _{\alpha }:=\inf_{{\u}\in \B_{ \X}}\mathbf{g}_{\alpha }({\u} ),\,
\end{align}
because then the operator that $\rho $ induces on the vector space spanned by constant functions and functions $\{\mathbf{g}_{\alpha }:\,\alpha \in \F_{\X}'\}$ is positive, and one can apply Kantorovitch and Riesz-Markov theorems (see \LM) (the necessity of (\ref{eq:intro-positivity}) is a straightforward consequence of the positivity of the mathematical expectation). In this context, $\rho $ is sometimes loosely referred to as being positive.

The inequalities (\ref{eq:intro-positivity}) are the linear inequalities determining the convex set $\U_{\X}$, and we show in this article how this approach, effective in \cite{She63} and \cite{Qui08}, can be fruitful.  Even if $\X$ is finite, the positivity conditions (\ref{eq:intro-positivity}) involve a priori an infinity of relations, and it is not even known what is the right-hand side of (\ref{eq:intro-positivity}) for general $\alpha $.  In the discrete setting and in a slightly different formulation, McMillan \cite{McM55} has proved that it is enough to satisfy (\ref{eq:intro-positivity}) for $\alpha $ \emph{corner positive}, and $\kappa_{\alpha }=0$, but the corner positivity is poorly understood; it has been numerically studied in \cite{Qui08} if $\X$ has less than $7$ elements. At the present time it is not possible to give an algorithmic procedure that determines the positivity of a given function $(\rho _{x,y})$, and it seems like a very difficult challenge with the tools available, therefore the consensus is to give necessary conditions as sharp as possible for the admissibility of $\rho $. This problem has been posed by MacMillan \cite{McM55} in the field of telecommunications. It is more or less implicit in many  articles, and has been to the author's knowledge first  addressed directly by Shepp \cite{She63}, and more recently by Quintanilla \cite{Qui08}. A series of works by Torquato and his coauthors (see   \cite{JiaStiTor07}, and \cite{Tor02} Sec.  2.2 and references therein), in the field of materials science, gather known necessary conditions and illustrate them in many $2$D and $3$D   models. This question was developed alongside in the field of geostatistics; Matheron \cite{Mat93} has found via arithmetic considerations a wide class of necessary conditions, that he has proven  to be sufficient if $\X$ has cardinality less or equal to $5$, and he has conjectured these conditions to be sufficient for any $\X$. Those conditions form the widest known class of necessary conditions used in related practical problems.
Some other authors do not attack frontally  this question,  but address  the realisability problem within some particular classes of models, e.g. Gaussian, mosaic, or boolean model (see   \cite{ChiDel99,Eme10,Lan02, Mas72}). In the companion paper \comp, we use the theoretical results of the present work to give a heuristic algorithm allowing to discard some inadmissible covariances, and other applications.\\
 
Our method is based on a direct study of $\U_{\X}$, which is a convex polytope if $\X$ has finite cardinality, and can therefore be applied  tools from convex geometry.
In Section \ref{sec:convex-geometry}, we provide some structural information on $\U_{\X}$, such as its dimension and extreme points, and from there on $\X$ is essentially assumed to have finite cardinality. As an application we give a bound for the number of different states of a random field realising an arbitrary admissible covariance. Section \ref{sec:matheron} revolves around the linear inequalities characterising $\U_{\X}$, i.e.  finding the hyperplanes supporting its facets if $\X$ has finite cardinality; we give new necessary conditions that are the support of the heuristic algorithm developed in \comp, and disprove numerically the conjecture of Matheron on the form of the supporting hyperplanes normals. Section \ref{sec:Un} is devoted to more theoretical facts about $\U_{\X}$, such as the elements of its boundary and its graph structure for faces of low dimensions.

\subsection{Realisability and convex geometry}
\label{sec:convex-geometry}

It is a fairly trivial fact that $\U_{\X}$ is convex, indeed if $X_{1}$ and $X_{2}$ are two unit fields of $\X$, then the segment $[\rho^{X_{1}}, \rho^{X_{2}}] $ is comprised in $ \U_{\X}$ because for every $t   \in[0,1]$, $t    \rho^{X_{1}}+(1-t    )\rho^{X_{2}}$ is the unit covariance of  
\begin{equation*}
X=
\begin{cases}
X_{1}$ if $B=1\\
X_{2}$ if $B=0
\end{cases}
\end{equation*}
where $B$ is an independent Bernoulli variable with parameter $t$.  The convexity   is exploited here to lay out some notation and basic facts about realisability. See the appendix at the end of the paper for notation and basic vocabulary  in convex geometry. 
 For a function $\v$ on $\X$, denote   $(\v\otimes \v)_{x,y}=\v_x\v_y, \v\otimes \v \in \F_{\X}$.
%A symmetric function under the form $v\otimes v$   is said to be under the \emph{ product form}. Under the canonical identification $F\otimes F$ is the indicator function of the cartesian product $F\times F$ for $F\in\FF$.  
If $X$ is a  unit field, then   $\rho^X=\E[X\otimes X]$, where the expectation is taken component-wise. 
 Furthermore the structure of  vertices  of $\U_{\X}$ can be made explicit.

\begin{proposition}
\label{prop:extremes-UN}
The extreme points of $\U_{\X}$ are the ${\u} \otimes {\u} $, ${\u} \in\B_\X$.
\end{proposition}

\begin{proof}
Take ${\u} \in \B_{\X}$, and assume that  ${\u} \otimes {\u} =t   \rho^{X_1}+(1-t   )\rho^{X_2}$ for some unit fields $X_1,X_2$ and $t   \in [0,1], t   \neq 0,1$. Since for every $x,y\in \X$ we have $t   \rho^{X_{1}}_{x,y}+(1-t   )\rho^{X_{2}}_{x,y}\in \{-1,1\}$, necessarily  $\rho^{X_{1}}_{x,y}=\rho^{X_{2}}_{x,y}\in \{-1,1\}$ for every pair $( {x,y})$ of $\X^{2}$. It follows that  both $X_{1}$ and $X_{2}$ are   deterministic with $\rho^{X_{1}}=\rho^{X_{2}}$, and therefore  ${\u}\otimes {\u}$ is an extreme point.

Conversely, take $\rho =\rho^{X}\in \U_{\X}$, assume that for some $x,y\in \X$, $t   =\rho^{X}_{x,y} \in (-1,1)$, meaning $\P(X_x\neq X_y)\notin\{0,1\}  $, and define the two unit fields $X_{1}$ and $X_{2}$ with respective laws 
\begin{align*}
\mathbb{P}(X_{1}\in A)=\mathbb{P}(X\in A~|~X_x\neq X_y)\\
\mathbb{P}(X_{2}\in A)=\mathbb{P}(X\in A~|~X_x=X_y)
\end{align*}
for $A$ a subset of $\B_{\X}$.
We have, by conditioning, 
\begin{equation*}
\rho^{X}=t    \rho^{X_{1}}+(1-t   )\rho^{X_{2}}.
\end{equation*}
It follows that $\rho^{X}$ lies in the relative interior of the segment $[\rho^{X_{1}},\rho^{X_{2}}]$ (this segment is not a singleton because $1=\rho^{X_{1}}_{x,y}\neq \rho^{X_{2}}_{x,y}=0$).
Thus  $\rho^{X}$ is not an extreme point of $\U_{\X}$.
\end{proof}
%\begin{corollary}
%\label{cor:star}
%The convex $\C_{N}$ is star-shaped around $0$: for every $c\in \C_{N}$ and $s\in [0,1]$, $sc\in \F_{N}$.
%\end{corollary}The next step in understanding $\C_{N}, \U_{N}$ and $\V_{N}$ is to find their intrinsic dimension and the number of their extreme points. We also identify $\F_{N} $ with the space   of symmetric $N\times N$ matrices, its subsets are plunged accordingly.  \\

In the sequel we focus on the finite case $\X=[N]=\{1,2,\dots ,N\}$, $N\geq 1$. For the sake of clarity index $[N]$ is replaced with $N$. Functions of $\F_{N}$ are  identified with $N\times N$ matrices.

\begin{remark}[Notation for triangular arrays]
Call $\pi$ the projection operator that takes a matrix $\rho\in \F_{N} $ to its supra-diagonal components $\pi(\rho)=(\rho_{i,j};\; 1 \leq i<j \leq N)$, and call $\F_{N}^*=\pi(\F_{N} )$ the space of supra-diagonal triangular arrays. An element $\rho$ of $\F_{N}^*$ is  represented as the triangular array 
\begin{equation*}
\left(\begin{array}{cccc} \quad & \rho_{ {1,2}}  & \dots &   \rho_{{ 1,N  }} \\&  & \ddots & \dots \\  & &   & \rho _{{N-1,N}} \\ & & &  \end{array}\right).
\end{equation*}\end{remark}

Since any $\rho \in \U_{N}$ automatically has the diagonal filled with $1$'s, it is somehow intuitive that mapping $\U_{N}$ to $\U_{N}^*=\pi(\U_{N})$ does not lose any relevant information. The dimension of $\F^*_{N}$ is $d_{N}:=\frac{N(N-1)}{2}$, and we prove below that $\U_{N}^*$ is a full-dimensional convex subset of $\F_{N}^*$.
The elements said to be under the \emph{product form} are those  that can be written, for some  ${\v} \in \mathbb{R}^N$,   
\begin{equation*}
{\v} \obslash {\v} :=\pi({\v} \otimes {\v} )=({\v} _i{\v} _j)_{1 \leq i< j \leq N}.
\end{equation*}
\begin{proposition}
\label{prop:dimension}
$\U_{N}^*$ has a non-empty interior in $\F_{N}^*$. Its extreme points are the $2^{N -1}$ vertices of the form ${\u} \obslash {\u} $ for $u\in \B_{N}$.
\end{proposition}

\begin{proof}
It suffices to prove that $\F_{N}^{*}$ contains at least $d_{N}$ linearly independent vectors. Take $i_{0}<j_{0}$ in $[N]$. Let $X_{k}, k\neq  j_{0}$ be independent  Rademacher variables  ($\P(X_{i}=1)=\P(X_{i}=-1)=1/2$) and define the unit field $X=(X_{1},\dots ,X_{j_{0}-1},X_{i_{0}},X_{j_{0}+1},\dots ,X_{N})$. Its unit covariance is the canonical vector $\mathsf{e}_{i_{0},j_{0}}=(1_{\{_{i=i_{0},j=j_{0}}\}})_{1 \leq i < j \leq N}$. Therefore   $\U_{N}^*$ contains $0$ and the $d_{N}$ such canonical vectors, and has nonempty interior in $\F_{N}^{*}$.  

In virtue of Proposition \ref{prop:extremes-UN} the extreme points of $\U_{N}$ are the $\u \otimes \u $ for $\u \in \B_{N}$, whence the extreme points of $\U_{N}^{*}=\pi (\U_{N})$ are of the form ${\u} \obslash {\u} , {\u} \in \B_{N}$. Conversely we can prove that each ${\u} \obslash {\u} $ is an extreme point by mimicking the arguments from the proof of Proposition \ref{prop:extremes-UN}, simply by assuming $x<y$. All the ${\u} \obslash {\u} $ are hence extreme points, and we have ${\u} \obslash {\u} ={\v} \obslash {\v} $ if and only if ${\u} ={\v} $ or ${\u} =-{\v} $. It follows that the number of extreme points is the number of binary vectors, under the identification ${\u} \equiv-{\u} $; we arrive at $2^{N-1}$.
\end{proof}

  A consequence of these  remarks is that the number of different  values taken by a random set realising a given unit covariance can be chosen to be no larger than $d_{N}+1$.

\begin{proposition}
Any unit covariance $\rho\in \U_{N}$ can be realised by a random field $X$ that takes at most $d_{N}+1$ distinct values.
\end{proposition}

\begin{proof}
Applying the Minkowski-Carath\'eodory theorem to $\U_{N}^*$ in $\F^*_{N}$,  $\pi(\rho)$ can be expressed as the convex combination of $d_{N}+1$ extreme points ${\u} ^{k}\obslash {\u} ^{k} , {\u} ^{k}\in \B_{N}$, $1\leq  k \leq d_{N}+1$ 
\begin{equation*}
 \pi(\rho)=\sum_{k=1}^{d_{N}+1}p_{k}{\u}^{k}\obslash {\u}^{k},
\end{equation*}
with $p_{k}\geq 0, \sum_{k}p_{k}=1$, meaning $\rho$ is the \uc~of the unit field $X$ which law is defined by 
\begin{equation*}
P(X={\u} ^{k})=p_{k}\in [0,1]
\end{equation*}
for every $1 \leq k \leq d_{N}+1$.
\end{proof}

\begin{remark}
\label{rmk:inradius}
Another consequence is that $\U_{N}^{*}$ contains a non-empty open ball. Using Gaussian unit covariances, it is proved in   \comp~ that the euclidean ball of $\F_{N}^{*}$ centred in $0$ with radius $\sqrt{2}/\pi $ is comprised in $\U_{N}^{*}$.
\end{remark}

%\query{  The proof of Th. \ref{thm:invariance}, as well as more general statements, can be derived from Th.~2.13 in \cite{LacMol}. Under this form, the problem of characterising numerically stationary covariances is the same without stationarity.  Nevertheless it might be possible to take advantage of the special Toeplitz form of  stationary field covariances to make efficient computations, for instance by imbedding it in a circulant matrix, as it  has been made in \cite{ChaWoo,DemMalShe}. For the same reasons, the numerical complexity of the problem reduces since the dimension of the space where live stationary covariances is smaller, but the combinatorial symmetry is lost, thus there is no gain  on the theoretical point of view.}
\section{Checking realisability numerically and Matheron's conjecture}
\label{sec:matheron}

This section focuses on the practical problem of checking, when $\X=[N], N \geq 1$, is finite, the validity of a given function $\rho $.  We saw in Proposition \ref{prop:dimension} that the problem is properly posed in $\F_{N}^*$, and since $\U_{N}^*$ is a polytope, it can be written as a finite intersection of half spaces, 
\begin{equation*}
\U_{N}^*=\bigcap_{k=1}^{h_{N}} W_{{\n}^{k}}^{o_{k}}
\end{equation*}
where $h_{N}$ is the number of  facets, the ${\n}^{k}\in \F_{N}^*$ are the outer normals of $\U_{N}^*$, the $o_{k}\in \mathbb{R}$ are the corresponding offset values, and \begin{equation*}
W_{{\n}^{k}}^{o_{k}}=\{\rho\in\F^*_{N}: \langle \rho,{\n}^{k}\rangle \leq o_{k}\}.
\end{equation*}

For a finite-dimensional vector space $V$, call $ \langle \cdot , \cdot    \rangle_{V}$ the canonical scalar product on $V$. All it takes to find $o_{k}$ knowing ${\n}^{k}$ is   to compute the infimum of $\langle {\n}^{k},\cdot\rangle_{\F_{N}^{*}}$ on the $2^{N-1}$ extreme points of $\U_{N}^*$ (prop. \ref{prop:dimension}).  The computational problem consists in finding the normals  ${\n}^{k}$, the complexity of computing the corresponding offsets is much smaller, see Table \ref{table}.
In \cite{Mat93}, Matheron studied the structure of the convex set constituted by realisable covariograms, similar to that of $\U_{N}$. For $N \leq  5$,
he computed by hand the normal vectors  and detected a recurring pattern;  he then conjectured that this form should be valid for every $N \geq 1$. We transposed and extended his results with numeric computations  here to \uc s.
Introduce the class 
\begin{align*}
\sE_{N}=\{{\e} \in\mathbb{Z}^N:\, \sum_{i=1}^{N}{\u} _{i}{\e} _{i }=1\text{ for some binary vector }{\u} \in \B_{N}=\{-1,1\}^N\}.\\
\end{align*} 

The following theorem gives a neat characterisation of unit covariances for $N\leq 6$.

\begin{theorem}
\label{thm:matheron-conditions}
\begin{itemize}
\item[(i)] Take $N\leq 6$.  Every outer normal ${\n}^{k}, 1\leq  k\leq h_{N}$, can be written under the form ${\e}^{k} \obslash{\e}^{k} $ for some ${\e} \in \sE_{N}$. A matrix $\rho\in\F_{N}^1$ belongs to $\U_{N}$ if and only if  for every  ${\e} \in\sE_{N}$,
\begin{equation}
\label{eq:Matheron}
 \sum_{i,j=1}^{N}\rho_{i j}{\e} _{i}{\e} _{j}\geq 1.
\end{equation}
\item[(ii)] Let $N \geq 1$, and $\rho\in\U_{N}$. For every ${\e} \in \mathbb{Z}^N$ with odd sum, $\rho$ satisfies (\ref{eq:Matheron}).
\end{itemize}

\end{theorem}

\begin{proof}
(i) The outer normals ${\n}^{k}$ of $\U_{N}^{*}$ can be listed by a linear programming algorithm which input is  the $2^{N-1}$ extreme points of $\U_{N}^{*}$ and output is the outer normals and corresponding offset values of the polytope spanned by  these extreme points. We used the  \texttt{cdd+}  algorithm \cite{Fuk}, the list of these normals can be found at \\{\texttt{http://www.math-info.univ-paris5.fr/~rlachiez/realisability}}. We then checked that each normal ${\n}^{k}$ could indeed be put under the form ${\e} ^{k} \obslash{\e} ^{k} $ for some ${\e} ^{k} \in \sE_{N}$.  Thus  $\rho \in \F_{N}^{1}$ is realisable if and only if for every $1\leq k\leq h_{N}$
\begin{align}
\label{eq:eq1}
\langle \pi (\rho ), {\e} ^{k}\obslash {\e} ^{k}\rangle_{\F_{N}^{*}} \geq \inf_{\U_{N}^*}\langle \cdot , {\e} ^{k}\obslash {\e} ^{k} \rangle_{\F_{N}^{*}}.\end{align}
Since the minimum of a linear form on a convex polytope is necessarily reached in an extreme point,  Prop. \ref{prop:dimension} yields that (\ref{eq:eq1}) is equivalent to   
\begin{align*}
\langle\pi ( \rho) , {\e} ^{k}\obslash {\e} ^{k}\rangle_{\F_{N}^{*}} \geq \inf_{{\u} \in \B_{N}}\langle {\u} \obslash {\u}   , {\e} ^{k}\obslash {\e} ^{k} \rangle_{\F_{N}^{*}}, 1\leq k\leq h_{N}.
\end{align*}
Using the fact that $\rho \in \F^{1}_{N}$, $\sum_{i=1}^{N}\rho _{ii}({\e} ^{k}_{i})^{2}=\sum_{i=1}^{N}{\u} _{i}{\u} _{i}({\e}  ^{k}_{i})^{2} $ for every ${\u} \in \B_{N},1\leq k\leq h_{N}$, the prior condition is therefore  equivalent to 
\begin{align*}
\langle \rho  ,{\e}  ^{k}\otimes {\e}  ^{k}  \rangle_{\F_{N}}\geq \inf_{u\in \B_{N}}\langle {\u} \otimes {\u}  ,{\e}  ^{k}\otimes {\e}  ^{k}  \rangle_{\F_{N}}=\inf_{{\u} \in \B_{N}}\left(\sum_{i=1}^{N}{\u} _{i}{\e}  ^{k}_{i}\right)^{2}= 1  
\end{align*}
because $\sum_{i=1}^{N}{\u} _{i}{\e}  ^{k}_{i}$ is an odd number, and the value $1$ is reached for the particular ${\u} \in \B_{N}$ such that $\sum_{i=1}^{N}{\e}  ^{k}_{i}{\u} _{i}=1$. Finally, a matrix $\rho \in \F_{N}^{1}$ indeed belongs to $\U_{N}$ if and only if 
\begin{align*}
\sum_{i,j=1}^{N}\rho _{ij}{\e}  ^{k}_{i}{\e}  ^{k}_{j}\geq 1\text{ for every }1\leq k\leq h_{N}.
\end{align*} The necessity is contained in point (ii).\\
(ii) If $\rho=\rho_{X}$ for some unit field $X$,  and ${\e} \in\mathbb{Z}^N$ is such that $\sum_{i}{\e} _{i} $ is an odd number,
\begin{equation*}
\langle \rho_{X},{\e} \otimes {\e} \rangle_{\F_{N}}=\E \left(\sum_{i}{\e} ^{k}X_{i}\right)^2 \geq 1
\end{equation*}
because $\sum_{i}{\e} _{i}X_{i}$ is a.s. an odd integer. \end{proof}

\begin{remark}
\begin{enumerate}
\item Even though $\sE_{N}$ is infinite, point  (i) provides a finite time procedure to check the realisability of a given matrix $\rho $, see the details in the companion paper \comp.
\item Any vector of $\sE_{N}$ has odd sum, thus point (ii) indeed applies to $\sE_{N}$.
\item Matheron only made computations up to $N=5$, for which he found that we can furthermore state that ${\e} _{i}\in \{-1,0,1\}$, but numerical computations showed  that this is no longer true for $N=6$.
\item The point (ii) also originates in \cite{Mat93}.
\item The point  (i) concerns the normal vectors to all the facets of the polytope $\U_{N}^{*}$. Matheron originally formulated his result only for the normal vectors of the facets touching the vertex $(1,\dots ,1)$. In the meantime he mentioned  the idea that this could be used for other vertices  by exploiting the combinatorial symmetry of unit covariances. \end{enumerate}
\end{remark}

We give in table \ref{table} the  number of facets for $N \leq 7$.
\begin{table}[h!]\centering\begin{tabular}{|c|c|c|}
\hline
$N$ & dimension $d_{N}$ & $h_{N}$\\
\hline 
3& 3 & 4\\
4 & 6 &16\\
5 & 10 & 56\\
6 & 15 &368\\
7 & 21 &116 764\\
\hline
\end{tabular} 
\caption{ number of facets of $\U_{N}^{*}$}
\label{table}
\end{table}
It explodes as the dimension increases.
  For $N=8$, the algorithm was stopped in dimension $d_{8}=28$  after that  $2~ 216~ 100$  outer normals were found, no more memory space being available. To go further in the computations, one has to go deeper in understanding the structure of $\U_{N}^*$. Section \ref{sec:Un} gathers some metric and topological facts about $\U_{N}^{*}$ that can help optimise the processing time and understand the structure of $\U_{N}^{*}$.

 Theorem \ref{thm:matheron-conditions} provides necessary conditions that are used in the companion paper \comp~to design an algorithm able to discard some spherical variograms as admissible covariograms, a recurrent  problem  in geostatistics (see \cite{Lan02}, Sec. 3.2.2).

\subsection{Matheron's conjecture is not true}
\label{sec:matheron-false}

Matheron's conjecture is appealing, as if it were true it would provide  a  procedure to efficiently  check the realisability of a matrix for any $N\geq 1$. We show in this section via theoretic and  numeric arguments that the conjecture fails at $N=   7$. Recall that convexity-related notation and vocabulary are introduced in the appendix.\\
 
The conjecture was originally stated for the covariogram of a random set $Y\subseteq [N]$
\begin{align*}
\gamma ^ {Y}_{i,j}=\frac{1}{2}\P(1_{Y} ( x)\neq 1_{Y} (y)),\quad i,j\in [N],
\end{align*}
where $1_{Y}$ is the indicator function of $Y$. 
Calling $ \V_{N}$ the class of admissible covariograms (convex for the same reasons than  $\U_{N}$), and $\V_{N}^{*}=\pi (\V_{N})$, Matheron conjectured that the outer normals of $\cone(0;\V_{N}^{*})$ are  of the form  ${\e}  \obslash {\e}  $ for ${\e}  $ in $\mathbb{Z} ^{N}$ with $\sum_{i=1}^{N}{\e}  _{i}=1$.  Now $\Phi(\gamma_{Y} ):=1-4\gamma_{Y} $ is the \uc~ of the unit field $X=2Y-1$, therefore easy computations show that the outer normals of $\cone(1;\U_{N}^{*})$ (where ${1}=(1,\dots ,1)=\Phi(0)$) 
are  those same  ${\e}  \obslash{\e}  $. 
 We prove below that the conjecture, even under a slightlier general form, is equivalent to Th. \ref{thm:matheron-conditions}.

\begin{remark}Matheron conjectured that conditions (\ref{eq:Matheron}) characterise matrices contained in $\cone(0;\V_{N}^{*}),$ more precisely that for $\gamma \in \F_{N},$
\begin{align}
\label{eq:true-matheron}
\gamma \in \cone(0;\V_{N}) \text{ if and only if }\sum_{1\leq i<j \leq N}\gamma _{ij}{\e}_{i}{\e}_{j}\leq 0
\end{align}
for every ${\e}\in \sE_{N}$, but not that  they characterise the whole convex $\V_{N}^*$ as can be  seen in the literature.

\end{remark}

\begin{theorem}For any $N\geq 1$, 
Matheron's conjecture is equivalent to the conjecture that a matrix $\rho\in\F_{N}^1$ is a unit covariance if and only if it satisfies (\ref{eq:Matheron}).
\end{theorem}

\begin{proof} Let $\gamma\in\F_{N}^0=\{\gamma \in \F_{N}:\gamma _{i,i}=0\text{ for all }i\in [N]\}$ and $\rho\in\F_{N}^1$ be linked up by $\rho= 1-4 \gamma$, hence $\rho$ is realisable as a \uc~iff $\gamma$ is realisable as an \iv.

Let us assume that Matheron's conjecture is true, and that $\rho$ satisfies (\ref{eq:Matheron}); we must prove that $\rho$ is realisable. In particular, for ${\e} \in\sE_{N}$ with unit sum (i.e. $\sum_{i}{\e} _{i}=1$)
\begin{equation*}
\sum_{1\leq i,j \leq N}{\e} _{i}{\e} _{j}\gamma_{ij}=\frac{1}{4}\sum_{1 \leq  i,j \leq N}{\e} _{i}{\e} _{j}(1-\rho_{ij}) \leq 0,
\end{equation*}
whence by Matheron's conjecture $\pi( \gamma)$ lies in $\cone(0;\V_{N}^{*})$. Applying $\Phi$ yields that  $\pi( \rho)$ is in $\cone(1; \U_{N}^{*})$. Let us call, for ${\u} \in \{-1,1\}^N$, $\theta^{{\u} }$ the transformation of $\F_{N}$ defined by 
\begin{equation*}
\theta^{{\u} }\mu =({\u} _{i}{\u} _{j}\mu _{ij})_{1\leq i,j \leq N}, \mu \in \F_{N}.
\end{equation*}
The rotated matrix $\theta^{{\u} } \rho$ also verifies 
\begin{equation*}
\sum_{ij}{\e} _{i}{\e} _{j}(\theta^{{\u} } \rho)_{ij} \geq 1
\end{equation*}
for ${\e} \in \mathbb{Z}^N$ with unit sum because $\rho$ satisfies (\ref{eq:Matheron}). It follows that $\pi (\theta^{{\u} } \rho)$ also belongs to $\cone(1;\U_{N}^{*})$, and therefore $\pi (\rho )\in \cone(u\obslash u; \U_{N}^{*})$. Thus in each extreme point ${\u} \obslash {\u} $ of $\U_{N}^{*}$, ${\u} \in \{-1,1\}^{N} $ (see Prop. \ref{prop:extremes-UN}), $ \pi (\rho)$ is contained in $\cone({\u} \otimes {\u} ; \U_{N}^{*})$. Thus $\pi (\rho)$ is in $\U_{N}^{*}$, meaning it is realisable. 

In the other direction, if (\ref{eq:Matheron}) characterises the realisability of \uc s, then $\cone({1};\U_{N}^{*})$ is characterised by inequalities in (\ref{eq:Matheron}) that become equalities if $\rho _{ij}=1$, i.e. inequalities such that $\sum_{i}{\e}  _{i}=1$. It follows by applying $\Phi^{-1}$ that $\cone(0;\V_{N}^{*})$ is characterised by (\ref{eq:true-matheron}), which is exactly Matheron's conjecture.

\end{proof}

In the light of the previous proposition, Matheron hence conjectured that every normal vector of $\U_{N}^*$ is under a product form $\v\obslash \v$. We give in the file \texttt{rho7.ine} the normals of $\U_{7}^{*}$ computed with \texttt{cdd+}. The following array $\tau \in\F_{N}^*$ is one of these normals.
\begin{equation*}
\tau = \left(\begin{array}{ccccccc} \quad &  -2 & 1 & 0 & 1 & 1 & 1 \\  &   & -2 & 1 & -2 & -2 & -3 \\  &   &   & -1 & 1 & 1 & 2 \\  &   &    &   & -1 & -1 & -2 \\  &   &   &   &   & 1 & 2 \\  &   &   &   &   &   & 2 \\  &   &   &   &   &   &  \end{array}\right)
\end{equation*}
 It is clear that $\tau $ is not under the product form $\v\obslash \v$ for some $\v\in \mathbb{R}^N$ because $\tau _{14}=0$, but neither its first line nor its $4$-th column are filled with $0$'s. In conclusion Matheron's conjecture is not true for $N=7$. We proved in a similar way that it was not true for $N=8$.
%satisfies $\langle T,\alpha\obslash \alpha\rangle=-6$ for $d_{7}=21$ linearly independant vectors $\alpha\obslash \alpha$, and $-6$ is indeed the minimum of $\langle T,\cdot\rangle$ on $R_{N}$, i.e. on its extremal points. (the $\alpha$ concerned correspond to $2,5,6,17,18,21,24,29,33,34,37,40,45,49,52,55,56,57,61,63,64$). 

\section{The structure of  $\U_{N}^*$}
\label{sec:Un}

The task of giving a tractable characterisation of realisable covariances, or \uc s, seems a very hard one, not to say impossible. This problem, combinatorial in nature, relies heavily on arguments from convex geometry to analyse $ \U_{N}^{*}$. We give in this section topological facts about $\U_{N}^*$, in order to better apprehend its structure both for a geometric and a graph-theoretic description. The convex body $\U_{N}^{*}$ bears some peculiar properties, which might lead to think
that the combinatorial problem of characterising $\U_{N}^{*}$ is better apprehended with this geometric approach, and stands on its own as an interesting theoretical problem.

With the current knowledge of $\U_{N}^{*}$ one should more rely on numeric considerations to test the validity of a given model. Unfortunately, as is apparent in Table \ref{table}, the complexity of such a task  explodes as the cardinality $N$ of $\X$ increases. Understanding better the structure of $\U_{N}^*$ can enable programmers to design a more adapted linear programming algorithm to find the outer normals of $\U_{N}^*$, which amounts to find the necessary and sufficient conditions of realisability.\\

Let us start by recalling that $\U_{N}^*$ is a polytope of $\F_{N}^*$   which all $2^{N-1}$ extreme points are vertices of the hypercube (Prop. \ref{prop:dimension} ). 
  For algorithms  based on ray-shooting queries it is interesting to have an interior point of $\U_{N}^*$, or of its dual. Here $0$ plays perfectly this role, and this can be quantified.

\begin{proposition}
Calling $B_{0}(r)$ the euclidean ball of $\F_{N}^{*}$ with radius $r$ centred in $0$, we have
\begin{equation*}
B_{0}(\sqrt{2}/\pi) \subseteq \U_{N}^* \subseteq B_{0}(\sqrt{d_{N}}).
\end{equation*}

\end{proposition}

\begin{proof}The first inclusion translates Remark \ref{rmk:inradius}, itself relying on classical results about Gaussian covariances and exploited in  \comp.  
 The second traduces the fact that the extreme points of $\U_{N}^*$ are vertices of the hypercube.  
 \end{proof}

The next statement locates the singular admissible covariances on the boundary of $\U_{N}^{*}$.

\begin{theorem}
\label{thm:singular}
\label{thm:singular-boundary}  (i)The unit covariance $\rho $ of a unit field $X$ is singular if and only if for some non-trivial real deterministic coefficients $\lambda _{1},\dots ,\lambda _{N}$ the components of $X$ satisfy
\begin{align*}
\sum_{k}\lambda _{k}X_{k}=0\, a.s..
\end{align*}
(ii) Singular \uc s lie on the boundary of $\U_{N}^*$.
\end{theorem}

\begin{proof}  (i) Let $X$ be a unit field and $\rho$ its unit covariance. As a semi-definite positive matrix,  
the singularity of $\rho $ is equivalent to the existence of a non-trivial family of scalar numbers $\{\lambda _{k}:\,1 \leq  k \leq  N\}$ such that 
\begin{align*}
\sum_{1\leq  k,j\leq N}\lambda _{k}\lambda _{j}\rho _{kj}=0,
\end{align*}
which exactly means that $\E( \sum_{k=1}^{N}\lambda _{k}X_{k})^{2}=0$. 

(ii) Let $\rho$ be a singular \uc, and $\lambda_{k}, k=1,\dots ,N$, non-trivial coefficients such that 
\begin{align*}
\sum_{1\leq k,j\leq N}\lambda _{k}\lambda _{j}\rho _{kj}=0,\quad 1\leq  j \leq  N.
\end{align*}
Two at least of the $\lambda _{k}$ are non zero (because $\rho $ has $1$'s on the diagonal), say $\lambda _{i_{1}}$ and $\lambda _{i_{2}}$, and put $\sigma ={\rm sign}(\lambda _{i_{1}}\lambda _{i_{2}})\in \{-1,1\}$. 
Let us take $\varepsilon  >0$ and 
\begin{equation*}
\begin{cases}
\rho_{ij}^\varepsilon  =\rho_{j}-\sigma \varepsilon \text{ if }\{i,j\}=\{i_{1},i_{2}\}  ,\\
\rho_{ij}^\varepsilon =\rho_{ij}$ otherwise$.
\end{cases}
\end{equation*}

Then if $\rho^{\varepsilon } $ was realisable by some unit field $X^\varepsilon  $, we would have 
\begin{equation*}
\E(\sum_{k=1}^{N}\lambda_{k}X_{k}^\varepsilon  )^2=\sum_{k,l=1}^{N}\lambda_{k}\lambda_{l}\rho^\varepsilon _{kl}=\sum_{k,l=1}^{N}\rho_{kl}\lambda_{k}\lambda_{l}-2\lambda_{1}\lambda_{2}\sigma\varepsilon   =-\lambda_{i_{1}}\lambda_{i_{2}}\sigma\varepsilon  <0,
\end{equation*}
which is impossible.
Thus there are unrealisable covariances arbitrarily close from $\rho$, it follows that $\rho$ is on the boundary of $\U_{N}^*$.
\end{proof}

\begin{remark}
It is conversely not true that every $\rho$ on the boundary is   singular (otherwise finding the outer normals of $\U_{N}^*$, and thus characterising realisable \uc s, would be easy). According to Th.~\ref{thm:singular}  it would mean that $\sum_{i}\lambda_{i}X_{i}=0$ a.s. for some non-trivial family $(\lambda _{i})$, but if for instance $X$ takes the values $(-1,1,1), (-1,-1,1), (1,-1,1)$ each with probability $1/3$ yields the non-singular \uc 
\begin{equation*}
\rho_{X}=\left(\begin{array}{ccc}1 & -1/3 & -1/3 \\-1/3 & 1 & -1/3 \\-1/3 & -1/3 & 1\end{array}\right).
\end{equation*}

On the other hand $\pi (\rho_{X})$ lies on the border of $\U_{4}^*$ because it is a convex combination of three out of the  $4$ vertices of the tetrahedron  $\U_{4}^*$ (see Prop.  \ref{prop:dimension} or \cite{Mat93} p.110).

\end{remark}

A manner of describing exhaustively the topological structure of $\U_{N}^*$ is to  explicit the hypergraph structure of its vertices. The following theorem states that for $k<N$, the $k$-th order hypergraph structure of $\U_{N}^*$ is complete, in the sense that the simplex formed by any $k$-tuple of extreme points is contained in $\U_{N}^{*}$'s boundary.

\begin{theorem}

For every $k<N$, and $k$-tuple ${\u}^{1},\dots,{\u} ^{k}$ of $\{-1,1\}^N$, the $k$-dimensional simplex formed by their respective \uc s $\rho_{{\u} ^{1}},\dots,\rho_{{\u} ^{k}}$ only contains singular matrices and thus lies on the boundary of $\U_{N}^*$. 

\end{theorem}

\begin{proof}

A covariance of the form
 $\rho=\sum_{i=1}^kp_{i}\rho_{\u^{i}}$ is the covariance of the unit field $X$ with law $\P(X=\u^{i})=p_{i}$, allowed to take values among the $k$ vectors $\u^{i}$. Since $k<N$, there exists a non-trivial family $\lambda _{1},\dots ,\lambda _{N}$ such that 
\begin{align*}
\sum_{j=1}^{N}\lambda _{j}\u^{i}_{j}=0
\end{align*}
 for all $i$, whence a.s. $\sum_{j}\lambda _{j}X_{j}=0$. It follows by Th. \ref{thm:singular}(i) that $\rho $  is singular. The conclusion comes by applying Th. \ref{thm:singular-boundary}.\end{proof}

In some sense $\U_{N}^*$ has the topological structure of the $d_{N}$-dimensional simplex as long as one only looks at dimensions strictly smaller than $N$. Higher dimensional facets are harder to explicit.

\section{Acknowledgements}

I am grateful to Ilya Molchanov, who through many discussions about Matheron's conjecture, contributed to this research. Vincent Delos helped me find my way around linear programming for the numerical issues.

%%%%%%%%

\section*{Appendix: Convex geometry}

Let $C$ be a convex set of a  vector space $V$. The \emph{extreme points} of $C$ are the $x\in  V$ such that, for every $y,z\in C$, if $x$ lies in the segment $[y,z]$ spanned by $y$ and $z$, then either $x=y$ or $x=z$. Denote them by ${\rm ext(C )}$. All the subsequent statements are only made for finite dimensional space $V=\mathbb{R}^{d}$.  If $X$ is a subset of $V$, the \emph{convex hull} of $X$ is the smallest closed convex set of $V$ containing $X$. The \emph{relative affine space} ${\rm span}(C )$ of a convex $C$ is the vector space spanned by $C$.  In the topology of ${\rm span}(C )$, $C$ has a non-empty interior and the \emph{relative dimension} of $C$ is defined as the dimension of ${\rm span}(C )$. Say that $C$ is fully dimensional if ${\rm span}(C)=\mathbb{R}^{d}$, equivalent to the fact that $C$ has a non-empty interior in $\mathbb{R}^{d}$. Given a point $x\in  C$, denote by  ${\rm cone}(x;C)$ the smallest affine convex convex cone with vertex $x$ containing $C$.  

\subsection*{Polytopes}

A polytope $P$ of $\mathbb{R}^{d}$ is the convex hull of a finite number of points $x_{1},\dots ,x_{q}$ of $\mathbb{R}^{d}$, denoted $P={\rm conv}(x_{1},\dots ,x_{q})$. For $0 \leq k<d$, a \emph{$k$-dimensional face} of $P$ is any set $S\subseteq \partial P$ such that $S$ is the convex intersection of $P$ with a affine hyperplane and the relative dimension of $S$ is exactly $k$. For any $0 \leq  k \leq  d-1$, the data of all its $k$-dimensional facets characterise $P$. Its $0$-facets are its extreme points, also called \emph{vertices}, and the corresponding description of the form  ${\rm ext}(P )=\{x_{1},\dots ,x_{q}\}$ is called $V$-description. At the opposite, its $(d-1)$-dimensional facets are sometimes just called \emph{facets}, and yield the so-called \emph{H-description}. The latter can be understood via a finite family of affine half-spaces $H_{1},\dots ,H_{q}$ where 
\begin{align*}
H_{i}=\{y\in \mathbb{R}^{d}:\, \langle y,{\n}_{i} \rangle \leq  o_{i} \}
\end{align*}
for some ${\n}_{i}\in  \mathbb{R}^{d}$ called \emph{outer normal} and corresponding \emph{offset value} $o_{i}$, $1 \leq  i \leq  q$. Given any $x\in  {\rm \ext}(P )$, the convex cone of $C$ generated by $x$ is obtained by retaining only half spaces with $x$ in their boundary
\begin{align*}
\cone(x;V)=\bigcap _{i:\langle x,{\n}_{i} \rangle=o_{i} }H_{i}.
\end{align*}
For $k<d$, call $k$-th order hypergraph of $C$ the class of $k$-tuples of extreme points of $C$ $\{x_{1},\dots ,x_{k}\}$ such that the simplex formed by $x_{1},\dots ,x_{k}$ lies on the boundary of $C$.

\end{document}